\documentclass{amsart}%
\usepackage{amsfonts}
\usepackage{amsmath}
\usepackage{amssymb}
\usepackage{graphicx}%
\providecommand{\U}[1]{\protect\rule{.1in}{.1in}}
\newtheorem{theorem}{Theorem}
\theoremstyle{plain}

\newtheorem{definition}{Definition}

\newtheorem{lemma}{Lemma}

\newtheorem{remark}{Remark}

\numberwithin{equation}{section}
\begin{document}
\title[Second order delayed discrete systems]{Explicit solution of second-order delayed discrete equations}
\author{Nazim I. Mahmudov$^{1,2}$}
\address{$^{1}$Department of Mathematics Eastern Mediterranean University \\
Famagusta, 99628, T.R. North Cyprus Mersin 10, Turkey\\
$^{2}$Research Center of Econophysics Azerbaijan State University\\
of Economics (UNEC) Istiqlaliyyat Str. 6, Baku 1001, Azerbaijan}
\email{nazim.mahmudov@emu.edu.tr}

\begin{abstract}
A system of inhomogeneous second-order difference equations with linear parts
given by noncommutative matrix coefficients are considered. Closed form of its
solution is derived by means of newly defined delayed matrix sine/cosine using
the $\mathcal{Z}$-transform and determining function.

\end{abstract}
\maketitle

\section{Introduction}

In what follows we use the following notations:

\begin{itemize}
\item $\Theta$ is a zero matrix, and $I$ is an identity matrix;

\item $Z_{a}^{b}:=\left\{  a,a+1,...,b\right\}  $ for $a,b\in Z\cup\left\{
\pm\infty\right\}  $, $a\leq b$;

\item $\mathbb{M}_{r\times p}$ is the space of $r\times p$ matrices;$\times$

\item An empty sum $%
%TCIMACRO{\dsum _{i=a}^{b}}%
%BeginExpansion
{\displaystyle\sum_{i=a}^{b}}
%EndExpansion
z\left(  i\right)  =0$ and an empty product $%
%TCIMACRO{\dprod _{i=a}^{b}}%
%BeginExpansion
{\displaystyle\prod_{i=a}^{b}}
%EndExpansion
z\left(  i\right)  =1$ for integers $a<b$, where $z\left(  i\right)  $ is a
given function which does not have to be defined for each $i\in Z_{b}^{a}$ in
this case;

\item $x\left(  t+1\right)  -x\left(  t\right)  =:\Delta x\left(  t\right)  $
is the forward difference operator;

\item $x\left(  t+2\right)  -2x\left(  t+1\right)  +x\left(  t\right)
=:\Delta^{2}x\left(  t\right)  .$
\end{itemize}

Approximately two decades ago, Diblik and Khusainov \cite{3a}, \cite{4a}
introduced explicit representations for solutions to discrete linear systems
with a single pure delay using delayed discrete exponential matrices. Later,
Khusainov et al. \cite{12a} extended this approach to derive analytical
solutions for oscillatory second-order systems with pure delays by introducing
delayed discrete sine and cosine matrices. These pioneering contributions
spurred significant advancements in the analytical solutions of retarded
integer and fractional differential equations, as well as delayed discrete
systems, as seen in References \cite{6a}, \cite{9a}, \cite{10a}, \cite{21a}.
Building on these results, Diblik and Mor\c{c}ivkova \cite{7a}, \cite{8a}
extended the analysis to discrete linear systems with two pure delays, while
Pospi\c{s}il \cite{19a} applied the $\mathcal{Z}$-transform to address
multi-delayed systems with linear components represented by permutable matrix
coefficients. In 2018, Mahmudov \cite{16a} provided explicit solutions for
discrete linear delayed systems with nonconstant coefficients and
nonpermutable matrices, including first-order differences. Mahmudov \cite{17a}
later generalized these findings, removing the singularity condition on the
non-delayed coefficient matrix and deriving explicit solutions using the
$\mathcal{Z}$-transform. Furthermore, Diblik and Mencakova \cite{6a} presented
closed-form solutions for purely delayed discrete linear systems with
second-order differences, while Elshenhab and Wang \cite{22a} recently
addressed explicit representations for second-order difference systems with
multiple pure delays and noncommutative coefficient matrices.

These studies have yielded numerous insights into the qualitative theory of
discrete delay systems, encompassing stability analysis, optimal control
theory, and iterative learning control, as highlighted in References
\cite{2a}, \cite{5a}, \cite{14a}, \cite{15a}, \cite{18a}, and \cite{20a}.

Although significant progress has been made in studying linear discrete
systems and linear delayed discrete systems, research on iterative learning
control for delayed linear discrete systems with higher-order differences
remains limited. 

Therefore, motivated by \cite{6a}, \cite{16a}, \cite{22a}, we consider an
explicit representation of solutions of the following discrete second-order
systems with a single delay,%
\begin{equation}
\Delta^{2}y\left(  t\right)  +Ay\left(  t\right)  +By\left(  t-m\right)
=f\left(  t\right)  ,\ \ t\in\mathbb{Z}_{0}^{\infty},\ m\in\mathbb{Z}%
_{1}^{\infty} \label{dis1}%
\end{equation}
where $m$ is a delay, $A,B\in\mathbb{M}_{d\times d}$, $y:\mathbb{Z}%
_{-m}^{\infty}\rightarrow\mathbb{R}^{d}$ is a solution, $f:\mathbb{Z}%
_{0}^{\infty}\rightarrow\mathbb{R}^{d}$ is a function.

Let $\varphi:\mathbb{Z}_{-m}^{1}\rightarrow\mathbb{R}^{d}$ be a function. We
attach to (\ref{dis1}) the following initial conditions%
\begin{equation}
y\left(  t\right)  =\varphi\left(  t\right)  ,\ \ t\in\mathbb{Z}_{-m}^{1}.
\label{dis2}%
\end{equation}
It is well known that the initial value problem (\ref{dis1}), (\ref{dis2}) has
a unique solution in $\mathbb{Z}_{-m}^{\infty}$.

\section{Delayed discrete matriy sine/cosine}

One of the tools in this study is the $\mathcal{Z}$-transform defined as%
\[
\mathcal{Z}\left\{  f\left(  t\right)  \right\}  \left(  z\right)  =%
%TCIMACRO{\dsum \limits_{t=0}^{\infty}}%
%BeginExpansion
{\displaystyle\sum\limits_{t=0}^{\infty}}
%EndExpansion
\frac{f\left(  t\right)  }{z^{t}}\ \ \text{for}z\in\mathbb{R}.
\]
The $\mathcal{Z}$-transform is considered component-wisely, i.e. the
$\mathcal{Z}$-transform of a vector-valued function is a vector of
$\mathcal{Z}$-transformed coordinates.

\begin{definition}
\label{def:1} We say that the function $f:\mathbb{Z}_{0}^{\infty}%
\rightarrow\mathbb{R}^{d}$ is eyponentially bounded if there eyists
$b_{1},b_{2}>0$ such that%
\[
\left\Vert f\left(  t\right)  \right\Vert \leq b_{1}b_{2}^{t}\ \ \ \text{for
}\ t\in\mathbb{Z}_{0}^{\infty}.
\]

\end{definition}

\begin{lemma}
\label{def:01} The $\mathcal{Z}$-transform $\mathcal{Z}\left\{  f\left(
t\right)  \right\}  \left(  z\right)  $ of eyponentially bounded function
$f:\mathbb{Z}_{0}^{\infty}\rightarrow\mathbb{R}^{d}$ eyists for all $z$
sufficiently large.
\end{lemma}

The next lemma gather up some features of the $\mathcal{Z}$-transform.

\begin{lemma}
\label{lem:1}Assume that $f_{1},f_{2}:\mathbb{Z}_{0}^{\infty}\rightarrow
\mathbb{R}^{d}$ are eyponentially bounded functions. The for sufficiently
large $z\in\mathbb{R}$ we have:

\begin{enumerate}
\item $\mathcal{Z}\left\{  af_{1}\left(  t\right)  +bf_{2}\left(  t\right)
\right\}  =a\mathcal{Z}\left\{  f_{1}\left(  t\right)  \right\}
+b\mathcal{Z}\left\{  f_{2}\left(  t\right)  \right\}  ,\ \ \ a,b\in
\mathbb{R};$

\item $\mathcal{Z}^{-1}\left\{  z^{-l}\right\}  \left(  t\right)
=\delta\left(  l,t\right)  \ $for $l\in Z_{0}^{\infty},$ where $\delta$ is the
Kroneker delta,%
\[
\delta\left(  l,t\right)  =\left\{
\begin{array}
[c]{c}%
1,\ \ \ t=l,\\
0,\ \ \ t\neq l.
\end{array}
\right.
\]

\item $\mathcal{Z}^{-1}\left\{  F_{1}\left(  z\right)  F_{2}\left(  z\right)
\right\}  \left(  t\right)  =\left(  f_{1}\ast f_{2}\right)  \left(  t\right)
$. Here the convolution operation $\ast$ is defined by%
\[
\left(  f\ast g\right)  \left(  t\right)  =\sum_{j=0}^{t}f\left(  j\right)
g\left(  t-j\right)  ;
\]

\item $\mathcal{Z}^{-1}\left\{  \frac{z}{z-1}\right\}  \left(  t\right)
=\sigma\left(  t\right)  \ $for $z>1$, here $\sigma$ is the step function
defined as%
\[
\sigma\left(  t\right)  =\left\{
\begin{array}
[c]{c}%
1,\ \ \ t\geq0,\\
0,\ \ \ t<0.
\end{array}
\right.
\]

\item $\mathcal{Z}^{-1}\left\{  \frac{1}{\left(  z-1\right)  ^{l}}\right\}
\left(  t\right)  =\left(
\begin{array}
[c]{c}%
t-1\\
l-1
\end{array}
\right)  \sigma\left(  t-l\right)  $, $l\in\mathbb{Z}_{0}^{\infty}$.

\item $\mathcal{Z}\left\{  f_{1}\left(  t+n\right)  \right\}  =z^{n}%
\mathcal{Z}\left\{  f_{1}\left(  t\right)  \right\}  -%
%TCIMACRO{\dsum \limits_{j=0}^{n-1}}%
%BeginExpansion
{\displaystyle\sum\limits_{j=0}^{n-1}}
%EndExpansion
f_{1}\left(  j\right)  z^{n-j}$, $n\in\mathbb{Z}_{0}^{\infty}.$
\end{enumerate}
\end{lemma}

We introduce the determining matrix equation for $Q\left(  t;s\right)  ,$
$t=1,2,...$
\begin{align*}
Q\left(  t+1;s\right)   &  =AQ\left(  t;s\right)  +BQ\left(  t;s-1\right)  ,\\
Q\left(  0;s\right)   &  =%
\begin{pmatrix}
0 & 0\\
0 & 0
\end{pmatrix}
,\ \ Q\left(  1;0\right)  =%
\begin{pmatrix}
1 & 0\\
0 & 1
\end{pmatrix}
,\\
t  &  =1,2,...10,\ s=1,2,..10.
\end{align*}
where $I$ is an identity matriy, $\Theta$ is a zero matriy.

\begin{remark}
\begin{enumerate}
\item Simple calculations show that%
\begin{equation}%
\begin{tabular}
[c]{|l|l|l|l|l|l|l|}\hline
& $s=0$ & $s=1$ & $s=2$ & $s=3$ & $\cdots$ & $s=p$\\\hline
$Q\left(  1;s\right)  $ & $I$ & $\Theta$ & $\Theta$ & $\Theta$ & $\cdots$ &
$\Theta$\\\hline
$Q\left(  2;s\right)  $ & $A$ & $B$ & $\Theta$ & $\Theta$ & $\cdots$ &
$\Theta$\\\hline
$Q\left(  3;s\right)  $ & $A^{2}$ & $AB+BA$ & $B^{2}$ & $\Theta$ & $\cdots$ &
$\Theta$\\\hline
$Q\left(  4;s\right)  $ & $A^{3}$ & $A\left(  AB+BA\right)  +BA^{2}$ &
$AB^{2}+B\left(  AB+BA\right)  $ & $B^{3}$ & $\cdots$ & \\\hline
$\cdots$ & $\cdots$ & $\cdots$ & $\cdots$ &  & $\cdots$ & $\Theta$\\\hline
$Q\left(  p+1;s\right)  $ & $A^{p}$ & $\cdots$ & $\cdots$ & $\cdots$ &
$\cdots$ & $B^{p}$\\\hline
$\vdots$ & $\vdots$ & $\vdots$ & $\vdots$ & $\vdots$ & $\vdots$ & $\ddots
$\\\hline
\end{tabular}
\label{tab}%
\end{equation}

\item If $A$ and $B$ are commutative, that is, $AB=BA$, then we have%
\[
\ Q\left(  t+1;j\right)  =\left(
\begin{array}
[c]{c}%
t\\
j
\end{array}
\right)  A^{t-j}B^{j}\sigma\left(  t-j\right)  .
\]

\item If $A=\Theta$, then%
\[
\ Q\left(  t+1;j\right)  =\left\{
\begin{tabular}
[c]{ll}%
$\Theta,$ & $t+1\neq j,$\\
$B^{j},$ & $t+1=j.$%
\end{tabular}
\right.  .
\]

\end{enumerate}
\end{remark}

\begin{definition}
The delayed discrete matrix $M_{c}(t,A,m)$ is defined as:
\[
M_{c}(t,A,m):=%
\begin{cases}
\Theta, & \text{if }t\in\mathbb{Z}_{-\infty}^{-m-1},\\
I, & \text{if }t\in\mathbb{Z}_{-m}^{1},\\
I-A\binom{t}{2}+A^{2}\binom{t-m}{4}-\cdots+(-1)^{l}A^{l}\binom{t-(l-1)m}%
{2l}, & \text{if }t\in\mathbb{Z}_{(l-1)(m+2)+2}^{l(m+2)+1},\,l=0,1,2,\ldots.
\end{cases}
\]
Here:

\begin{itemize}
\item $\Theta$ represents the zero matrix.

\item $I $ is the identity matrix.

\item $\binom{a}{b} $ denotes the binomial coefficient, defined as $\binom
{a}{b} = \frac{a!}{b!(a-b)!} $, with $\binom{a}{b} = 0 $ if $b > a $ or $a < 0
$.
\end{itemize}
\end{definition}

\begin{definition}
\cite{6a}The delayed discrete matrix $M_{s}(t,A,m)$ is defined as:
\[
M_{s}(t,A,m):=%
\begin{cases}
\Theta, & \text{if }t\in\mathbb{Z}_{-\infty}^{-m},\\
I\binom{t+m}{1}, & \text{if }t\in\mathbb{Z}_{-m+1}^{2},\\
I\binom{t+m}{1}-A\binom{t}{3}+A^{2}\binom{t-m}{5}-\cdots+(-1)^{l}A^{l}%
\binom{t-(l-1)m}{2l+1}, & \text{if }t\in\mathbb{Z}_{(l-1)(m+2)+3}%
^{l(m+2)+2},\,l=0,1,2,\ldots.
\end{cases}
\]

\end{definition}

\begin{definition}
\cite{6a}Delayed discrete matrix sine/cosine is defined as follows%
\begin{align*}
\text{Sin}\ ^{A,B}\left(  t\right)   &  :=\sum_{l=0}^{\left\lfloor \frac
{t-1}{m+2}\right\rfloor }{\sum\limits_{0\leq i\leq l}}\left(  -1\right)
^{l}\left(
\begin{array}
[c]{c}%
t-im\\
2l+1
\end{array}
\right)  Q\left(  l+1;i\right)  :\mathbb{Z}_{0}^{\infty}\rightarrow
\mathbb{M}_{n\times n},\\
\text{Cos}^{A,B}\left(  t\right)   &  :=\sum_{l=0}^{\left\lfloor \frac{t}%
{m+2}\right\rfloor }{\sum\limits_{0\leq i\leq l}}\left(  -1\right)
^{l}\left(
\begin{array}
[c]{c}%
t-im\\
2l
\end{array}
\right)  Q\left(  l+1;i\right)  :\mathbb{Z}_{0}^{\infty}\rightarrow
\mathbb{M}_{n\times n}.
\end{align*}

\end{definition}

\begin{remark}
If $A=\Theta$, then%
\begin{align*}
\text{Sin}\ ^{\Theta,B}\left(  t+m\right)   &  :=\sum_{l=0}^{\left\lfloor
\frac{t+m-1}{m+2}\right\rfloor }\left(  -1\right)  ^{l}\left(
\begin{array}
[c]{c}%
t+m-lm\\
2l+1
\end{array}
\right)  B^{l}=M_{s}(t,B,m),\\
\text{Cos\ }^{\Theta,B}\left(  t+m\right)   &  :=\sum_{l=0}^{\left\lfloor
\frac{t+m}{m+2}\right\rfloor }\left(  -1\right)  ^{l}\left(
\begin{array}
[c]{c}%
t+m-lm\\
2l
\end{array}
\right)  B^{l}=M_{c}(t,B,m).
\end{align*}

\end{remark}

\begin{lemma}
[Binomial formula for noncommuative matrices]\label{lem:q1}Let $A,B\in
\mathbb{M}_{d\times d}$ be two noncommutative matrices. Then for any
$t\in\mathbb{Z}_{0}^{\infty}$ we have
\begin{equation}
\left(  A+B\right)  ^{t}=%
%TCIMACRO{\dsum \limits_{i=0}^{t}}%
%BeginExpansion
{\displaystyle\sum\limits_{i=0}^{t}}
%EndExpansion
Q\left(  t+1;i\right)  . \label{b1}%
\end{equation}

\end{lemma}

\begin{proof}
From table (\ref{tab}) it can be easily seen that for $t=0,1,2$ the identity
(\ref{b1}) is true. Now we use induction, assuming that (\ref{b1}) is true for
$t=n$ we prove it for $t=n+1$
\begin{align*}
\left(  A+B\right)  ^{n+1}  &  =\left(  A+B\right)  {\sum\limits_{i=0}^{n}%
}Q\left(  n+1;i\right) \\
&  ={\sum\limits_{i=0}^{n}}AQ\left(  n+1;i\right)  +{\sum\limits_{i=0}^{n}%
}BQ\left(  n+1;i\right) \\
&  ={\sum\limits_{i=0}^{n}}AQ\left(  n+1;i\right)  +{\sum\limits_{i=1}^{n+1}%
}BQ\left(  n+1;i-1\right) \\
&  ={\sum\limits_{i=0}^{n+1}}A\left[  Q\left(  n+1;i\right)  +BQ\left(
n+1;i-1\right)  \right] \\
&  ={\sum\limits_{i=0}^{n+1}}Q\left(  n+2;i\right)  .
\end{align*}
Here we used the property $Q\left(  n+1;n+1\right)  =\Theta=Q\left(
n+1;-1\right)  .$
\end{proof}

\begin{lemma}
[Gronwall inequality]\cite{1a} Let%
\[
y\left(  t\right)  \leq b\left(  t\right)  +a\left(  t\right)  {\sum
\limits_{j=0}^{t-1}}f\left(  j\right)  y\left(  j\right)  ,\ \ t\in
\mathbb{Z}_{0}^{\infty}.
\]
Then%
\[
y\left(  t\right)  \leq b\left(  t\right)  +a\left(  t\right)  {\sum
\limits_{j=0}^{t-1}}b\left(  j\right)  f\left(  j\right)  {\prod
\limits_{i=j+1}^{t-1}}\left(  1+a\left(  j\right)  f\left(  i\right)  \right)
,\ \ t\in\mathbb{Z}_{0}^{\infty}.
\]

\end{lemma}

\begin{lemma}
\label{lem:2}For any $t\in\mathbb{Z}_{0}^{\infty}$ we have the following
identities%
\begin{align*}
\left(  A+B\sigma^{-m}\right)  ^{t}  &  ={\sum\limits_{\substack{i\left(
m+1\right)  \leq t\\i\geq0}}}Q\left(  t+1-im;i\right)  ,\\
\ \ \ \left(  A+Bz^{-m}\right)  ^{t}  &  ={\sum\limits_{\substack{i\left(
m+1\right)  \leq t\\i\geq0}}}z^{-im}Q\left(  t+1;i\right)  .\\
\mathcal{Z}^{-1}\left\{  \left(  \left(  z-1\right)  ^{2}+A+\frac{B}{z^{m}%
}\right)  ^{-1}\right\}   &  =\sum_{l=0}^{\infty}{\sum\limits_{0\leq i\leq l}%
}\left(  -1\right)  ^{l}\left(
\begin{array}
[c]{c}%
t-im-1\\
2l+1
\end{array}
\right)  Q\left(  l+1;i\right)  ,\\
\mathcal{Z}^{-1}\left\{  \frac{1}{z^{j+m}}\left(  (z-1)^{2}+A+\frac{B}{z^{m}%
}\right)  ^{-1}\right\}   &  =\sum_{l=0}^{\infty}\sum_{0\leq i\leq l}%
(-1)^{l}\binom{t-j-m-im-1}{2l+1}Q(l+1;i),\\
\mathcal{Z}^{-1}\left\{  z\left(  z-1\right)  \left(  \left(  z-1\right)
^{2}+A+\frac{B}{z^{m}}\right)  ^{-1}\right\}   &  =\sum_{l=0}^{\infty}%
{\sum\limits_{0\leq i\leq l}}\left(  -1\right)  ^{l}\left(
\begin{array}
[c]{c}%
t-im\\
2l
\end{array}
\right)  Q\left(  l+1;i\right)  .
\end{align*}

\end{lemma}

\begin{proof}
The first two identities follows from Lemma \ref{lem:q1}. We start with the
identity:
\[
\left(  I-C\right)  ^{j}\sum_{t=0}^{\infty}\binom{t+j-1}{j-1}C^{t}%
=I,\quad\text{where }\Vert C\Vert<1.
\]
For sufficiently large $z\in\mathbb{R}$ such that
\[
\left\Vert \frac{A}{(z-1)^{2}}+\frac{B}{z^{m}(z-1)^{2}}\right\Vert <1,
\]
we derive:
\begin{align*}
\left(  z^{2}-2z+1+A+\frac{B}{z^{m}}\right)  ^{-1}  &  =\frac{1}{(z-1)^{2}%
}\left(  I+\frac{A}{(z-1)^{2}}+\frac{B}{z^{m}(z-1)^{2}}\right)  ^{-1}\\
&  =\frac{1}{(z-1)^{2}}\sum_{j=0}^{\infty}(-1)^{j}\left(  \frac{A}{(z-1)^{2}%
}+\frac{B}{z^{m}(z-1)^{2}}\right)  ^{j}\\
&  =\frac{1}{(z-1)^{2}}\sum_{t=0}^{\infty}\frac{(-1)^{t}}{(z-1)^{2t}}\left(
A+\frac{1}{z^{m}}B\right)  ^{t}.
\end{align*}

Next we use the formulas:
\[
\left(  A+Bz^{-m}\right)  ^{t}=\sum_{0\leq i\leq t}z^{-im}Q(t+1;i),
\]
\[
\sum_{j=0}^{t}\delta(im,j)\binom{t-j-1}{2l+1}=\binom{t-im-1}{2l+1},
\]
and
\[
\delta(j+m+im,t)\ast\binom{t-1}{2l-1}=\binom{t-j-m-im-1}{2l-1}.
\]

Now, consider the inverse $\mathcal{Z}$-transform of the original series:
\begin{align*}
\mathcal{Z}^{-1}\left\{  \frac{1}{(z-1)^{2}}\sum_{l=0}^{\infty}\frac{(-1)^{l}%
}{(z-1)^{2l}}\left(  A+\frac{1}{z^{m}}B\right)  ^{l}\right\}   &  =\sum
_{l=0}^{\infty}\mathcal{Z}^{-1}\left\{  \frac{(-1)^{l}}{(z-1)^{2l+2}}\left(
A+\frac{1}{z^{m}}B\right)  ^{l}\right\} \\
&  =\sum_{l=0}^{\infty}\mathcal{Z}^{-1}\left\{  \sum_{0\leq i\leq l}%
\frac{(-1)^{l}}{z^{im}(z-1)^{2l+2}}Q(l+1;i)\right\} \\
&  =\sum_{l=0}^{\infty}\sum_{0\leq i\leq l}(-1)^{t}\delta(im,t)\ast\binom
{t-1}{2l+1}Q(l+1;i)\\
&  =\sum_{l=0}^{\infty}\sum_{0\leq i\leq l}(-1)^{l}\binom{t-im-1}%
{2l+1}Q(l+1;i).
\end{align*}

Using similar steps, for $A_{j}(t) $, we find:
\begin{align*}
A_{j}(t)  &  = \mathcal{Z}^{-1} \left\{  \frac{1}{z^{j+m}} \left(  (z-1)^{2} +
A + \frac{B}{z^{m}}\right)  ^{-1} \right\} \\
&  = \sum_{l=0}^{\infty}\sum_{0 \leq i \leq l} (-1)^{l} \delta(j + m + im, t)
\ast\binom{t-1}{2l+1} Q(l+1; i)\\
&  = \sum_{l=0}^{\infty}\sum_{0 \leq i \leq l} (-1)^{l} \binom{t-j-m-im-1}%
{2l+1} Q(l+1; i).
\end{align*}

For the third $\mathcal{Z}^{-1}$-transform, we have
\begin{align*}
\mathcal{Z}^{-1}\left\{  \frac{z}{z-1}\sum_{l=0}^{\infty}\frac{(-1)^{l}%
}{(z-1)^{2l}}\left(  A+\frac{1}{z^{m}}B\right)  ^{l}\right\}   &  =\sum
_{l=0}^{\infty}\mathcal{Z}^{-1}\left\{  \frac{(-1)^{l}}{(z-1)^{2l+1}}\left(
A+\frac{1}{z^{m}}B\right)  ^{l}\right\} \\
&  =\sum_{l=0}^{\infty}\sum_{0\leq i\leq l}(-1)^{l}\mathcal{Z}^{-1}\left\{
\frac{1}{z^{im-1}(z-1)^{2l+1}}\right\}  Q(l+1;i)\\
&  =\sum_{l=0}^{\infty}\sum_{0\leq i\leq l}(-1)^{l}\delta(im-1,t)\ast
\binom{t-1}{2l}Q(l+1;i)\\
&  =\sum_{l=0}^{\infty}\sum_{0\leq i\leq l}(-1)^{l}\binom{t-im}{2l}Q(l+1;i).
\end{align*}

\end{proof}

\begin{definition}
Delayed discrete matrix sine/cosine is defined as follows%
\begin{align*}
\text{Sin}\ ^{A,B}\left(  t\right)   &  :=\sum_{l=0}^{\infty}{\sum
\limits_{0\leq i\leq l}}\left(  -1\right)  ^{l}\left(
\begin{array}
[c]{c}%
t-im\\
2l+1
\end{array}
\right)  Q\left(  l+1;i\right)  :\mathbb{Z}_{0}^{\infty}\rightarrow
\mathbb{M}_{n\times n},\\
\text{Cos}^{A,B}\left(  t\right)   &  :=\sum_{l=0}^{\infty}{\sum\limits_{0\leq
i\leq l}}\left(  -1\right)  ^{l}\left(
\begin{array}
[c]{c}%
t-im\\
2l
\end{array}
\right)  Q\left(  l+1;i\right)  :\mathbb{Z}_{0}^{\infty}\rightarrow
\mathbb{M}_{n\times n}%
\end{align*}

\end{definition}

\begin{lemma}
For all $t\in\mathbb{Z}_{0}^{\infty},$one has%
\begin{equation}
\left\Vert \text{Sin}\ ^{A,B}\left(  t\right)  \right\Vert \leq l_{s}\left(
t\right)  ,\ \ \ \left\Vert \text{Cos}\ ^{A,B}\left(  t\right)  \right\Vert
\leq l_{c}\left(  t\right)  . \label{q7}%
\end{equation}

\end{lemma}

\begin{proof}
We only the first inequality
\begin{align*}
\left\Vert \text{Sin}\ ^{A,B}\left(  t\right)  \right\Vert  &  \leq\sum
_{l=0}^{\left\lfloor \frac{t-m-1}{m+2}\right\rfloor }{\sum\limits_{0\leq i\leq
l}}\left(
\begin{array}
[c]{c}%
t-im\\
2l+1
\end{array}
\right)  \left\Vert Q\left(  l+1;i\right)  \right\Vert \\
&  \leq\sum_{l=0}^{\left\lfloor \frac{t-m-1}{m+2}\right\rfloor }%
{\sum\limits_{0\leq i\leq l}}\left(
\begin{array}
[c]{c}%
t-im\\
2l+1
\end{array}
\right)  \left(
\begin{array}
[c]{c}%
l\\
i
\end{array}
\right)  \left\Vert A\right\Vert ^{l-i}\left\Vert B\right\Vert ^{i}\\
&  \leq l_{s}\left(  t\right)  .
\end{align*}

\end{proof}

\begin{lemma}
\label{lem:eb}Under exponential boundedness of $f:\mathbb{Z}_{0}^{\infty
}\rightarrow\mathbb{R}^{d},$ a solution of (\ref{dis1}), (\ref{dis2}) has the
same property, that is, exponentially bounded.
\end{lemma}

\begin{proof}%
\begin{align*}
{\sum}\Delta y\left(  t+1\right)   &  =\Delta y\left(  t\right)  -Ay\left(
t\right)  -By\left(  t-m\right)  +f\left(  t\right) \\
{\sum_{j=0}^{r-1}}\Delta y\left(  j+1\right)   &  ={\sum_{j=0}^{r-1}}\Delta
y\left(  j\right)  -{\sum_{j=0}^{r-1}}Ay\left(  j\right)  -{\sum_{j=0}^{r-1}%
}By\left(  j-m\right)  +{\sum_{j=0}^{r-1}}f\left(  j\right) \\
\Delta y\left(  r\right)   &  =\Delta\psi\left(  0\right)  -{\sum_{j=0}^{r-1}%
}Ay\left(  j\right)  -{\sum_{j=0}^{r-1}}By\left(  j-m\right)  +{\sum
_{j=0}^{r-1}}f\left(  j\right)
\end{align*}
Summing the above equality from $0$ to $t-1$, we obtain%
\[%
%TCIMACRO{\dsum _{r=0}^{t-1}}%
%BeginExpansion
{\displaystyle\sum_{r=0}^{t-1}}
%EndExpansion
\Delta y\left(  r\right)  =%
%TCIMACRO{\dsum _{r=0}^{t-1}}%
%BeginExpansion
{\displaystyle\sum_{r=0}^{t-1}}
%EndExpansion
\Delta\psi\left(  0\right)  -%
%TCIMACRO{\dsum _{r=0}^{t-1}}%
%BeginExpansion
{\displaystyle\sum_{r=0}^{t-1}}
%EndExpansion%
%TCIMACRO{\dsum _{j=0}^{r-1}}%
%BeginExpansion
{\displaystyle\sum_{j=0}^{r-1}}
%EndExpansion
Ay\left(  j\right)  -%
%TCIMACRO{\dsum _{r=0}^{t-1}}%
%BeginExpansion
{\displaystyle\sum_{r=0}^{t-1}}
%EndExpansion%
%TCIMACRO{\dsum _{j=0}^{r-1}}%
%BeginExpansion
{\displaystyle\sum_{j=0}^{r-1}}
%EndExpansion
By\left(  j-m\right)  +%
%TCIMACRO{\dsum _{r=0}^{t-1}}%
%BeginExpansion
{\displaystyle\sum_{r=0}^{t-1}}
%EndExpansion%
%TCIMACRO{\dsum _{j=0}^{r-1}}%
%BeginExpansion
{\displaystyle\sum_{j=0}^{r-1}}
%EndExpansion
f\left(  j\right)
\]
equivalently%
\[
y\left(  t\right)  =\psi\left(  0\right)  +t\Delta\psi\left(  0\right)  -%
%TCIMACRO{\dsum _{j=0}^{t-1}}%
%BeginExpansion
{\displaystyle\sum_{j=0}^{t-1}}
%EndExpansion
\left(  t-j\right)  Ay\left(  j\right)  -%
%TCIMACRO{\dsum _{j=0}^{t-1}}%
%BeginExpansion
{\displaystyle\sum_{j=0}^{t-1}}
%EndExpansion
\left(  t-j\right)  By\left(  j-m\right)  +%
%TCIMACRO{\dsum _{j=0}^{t-1}}%
%BeginExpansion
{\displaystyle\sum_{j=0}^{t-1}}
%EndExpansion
\left(  t-j\right)  f\left(  j\right)  .
\]
Taking the norm and applying the triangle inequality, we get%
\begin{align*}
\left\Vert y\left(  t\right)  \right\Vert  &  \leq\left\Vert \psi\left(
0\right)  \right\Vert +t\left\Vert \Delta\psi\left(  0\right)  \right\Vert
+{\sum_{j=0}^{t-1}}\left(  t-j\right)  \left\Vert A\right\Vert \left\Vert
y\left(  j\right)  \right\Vert \\
&  +%
%TCIMACRO{\dsum _{j=0}^{t-1}}%
%BeginExpansion
{\displaystyle\sum_{j=0}^{t-1}}
%EndExpansion
\left(  t-j\right)  \left\Vert B\right\Vert \left\Vert y\left(  j-m\right)
\right\Vert +%
%TCIMACRO{\dsum _{j=0}^{t-1}}%
%BeginExpansion
{\displaystyle\sum_{j=0}^{t-1}}
%EndExpansion
\left(  t-j\right)  \left\Vert f\left(  j\right)  \right\Vert \\
&  \leq\left\Vert \psi\left(  0\right)  \right\Vert +t\left\Vert \Delta
\psi\left(  0\right)  \right\Vert +%
%TCIMACRO{\dsum _{j=-m}^{-1}}%
%BeginExpansion
{\displaystyle\sum_{j=-m}^{-1}}
%EndExpansion
\left(  t-m-j\right)  \left\Vert B\right\Vert \left\Vert \psi\left(  j\right)
\right\Vert \\
&  +%
%TCIMACRO{\dsum _{j=0}^{t-1}}%
%BeginExpansion
{\displaystyle\sum_{j=0}^{t-1}}
%EndExpansion
\left(  t-j\right)  \left\Vert A\right\Vert \left\Vert y\left(  j\right)
\right\Vert +%
%TCIMACRO{\dsum _{j=0}^{t-m-1}}%
%BeginExpansion
{\displaystyle\sum_{j=0}^{t-m-1}}
%EndExpansion
\left(  t-m-j\right)  \left\Vert B\right\Vert \left\Vert y\left(  j\right)
\right\Vert \\
&  +%
%TCIMACRO{\dsum _{j=0}^{t-1}}%
%BeginExpansion
{\displaystyle\sum_{j=0}^{t-1}}
%EndExpansion
\left(  t-j\right)  \left\Vert f\left(  j\right)  \right\Vert .
\end{align*}
In thsi stage, without losing the generality, it is assumed that $b_{2}>1.$
Then%
\[%
%TCIMACRO{\dsum _{j=0}^{t-1}}%
%BeginExpansion
{\displaystyle\sum_{j=0}^{t-1}}
%EndExpansion
\left(  t-j\right)  \left\Vert f\left(  j\right)  \right\Vert \leq%
%TCIMACRO{\dsum _{j=0}^{t-1}}%
%BeginExpansion
{\displaystyle\sum_{j=0}^{t-1}}
%EndExpansion
\left(  t-j\right)  b_{1}b_{2}^{j}\leq\frac{t\left(  t+1\right)  }{2}%
b_{1}b_{2}^{t}.
\]
Thus%
\begin{align*}
\left\Vert y\left(  t\right)  \right\Vert  &  \leq a\left(  t\right)
+t\left(  \left\Vert A\right\Vert +\left\Vert B\right\Vert \right)
%TCIMACRO{\dsum _{j=0}^{t-1}}%
%BeginExpansion
{\displaystyle\sum_{j=0}^{t-1}}
%EndExpansion
\left\Vert y\left(  j\right)  \right\Vert .\\
b\left(  t\right)   &  :=\left\Vert \psi\left(  0\right)  \right\Vert
+t\left\Vert \Delta\psi\left(  0\right)  \right\Vert +%
%TCIMACRO{\dsum _{j=-m}^{-1}}%
%BeginExpansion
{\displaystyle\sum_{j=-m}^{-1}}
%EndExpansion
\left(  t-m-j\right)  \left\Vert B\right\Vert \left\Vert \psi\left(  j\right)
\right\Vert \\
&  +\frac{t\left(  t+1\right)  }{2}b_{1}b_{2}^{t}%
\end{align*}
From the Gronwall inequality%
\begin{align*}
\left\Vert y\left(  t\right)  \right\Vert  &  \leq b\left(  t\right)
+t\left(  \left\Vert A\right\Vert +\left\Vert B\right\Vert \right)
%TCIMACRO{\dsum _{j=0}^{t-1}}%
%BeginExpansion
{\displaystyle\sum_{j=0}^{t-1}}
%EndExpansion
b\left(  j\right)  {\prod_{i=j+1}^{t-1}}\left(  1+j\left(  \left\Vert
A\right\Vert +\left\Vert B\right\Vert \right)  \right) \\
&  \leq b\left(  t\right)  +t\left(  \left\Vert A\right\Vert +\left\Vert
B\right\Vert \right)
%TCIMACRO{\dsum _{j=0}^{t-1}}%
%BeginExpansion
{\displaystyle\sum_{j=0}^{t-1}}
%EndExpansion
b\left(  j\right)  \left(  1+t\left(  \left\Vert A\right\Vert +\left\Vert
B\right\Vert \right)  \right)  ^{t}\\
&  \leq b\left(  t\right)  \left(  1+t^{2}\left(  \left\Vert A\right\Vert
+\left\Vert B\right\Vert \right)  \left(  1+t\left(  \left\Vert A\right\Vert
+\left\Vert B\right\Vert \right)  \right)  ^{t}\right)  .
\end{align*}
Therefore, one can easily see that there eyists constants $\widehat{b}%
_{1},\widehat{b}_{2}>0$ such that
\[
\left\Vert y\left(  t\right)  \right\Vert \leq\widehat{b}_{1}\widehat{b}%
_{2}^{t},\ \ \ t\in\mathbb{Z}_{0}^{\infty}.
\]

\end{proof}

\section{Explicit solutions}

Below we state and prove the main theorem of the paper. The main instrument
used is the $\mathcal{Z}$-transform. We give a closed analytical form of the
solution of problem (\ref{dis1}), (\ref{dis2}) in terms of the delayed
discrete matrix sine/cosine.

\begin{theorem}
Let $f:\mathbb{Z}_{0}^{\infty}\rightarrow\mathbb{R}^{d}$ be an exponentially
bounded function. The solution $y(t)$ of the IVP problem (\ref{dis1}),
(\ref{dis2}) has the following form%
\begin{align}
y\left(  t\right)   &  =\text{Cos}^{A,B}\left(  t\right)  \varphi\left(
0\right)  +\text{Sin}^{A,B}\left(  t\right)  \Delta\varphi\left(  0\right)
\nonumber\\
&  +%
%TCIMACRO{\dsum \limits_{i=-m}^{-1}}%
%BeginExpansion
{\displaystyle\sum\limits_{i=-m}^{-1}}
%EndExpansion
\text{Sin}^{A,B}\left(  t-i-m-1\right)  B\varphi\left(  i\right)  +\sum
_{j=0}^{t-2}\text{Sin}^{A,B}\left(  t-j-1\right)  f\left(  j\right)  ,
\label{rep1}%
\end{align}
for $t\in Z_{2}^{\infty}.$
\end{theorem}

\begin{proof}
We recall that Lemma \ref{lem:eb} says that the $\mathcal{Z}$-transform of the
solution of (\ref{dis1}) exists. Therefore, one can apply the $\mathcal{Z}%
$-transform to both sides of the delayed system (\ref{dis1}) to get%
\begin{gather*}%
%TCIMACRO{\dsum \limits_{t=0}^{\infty}}%
%BeginExpansion
{\displaystyle\sum\limits_{t=0}^{\infty}}
%EndExpansion
\frac{y\left(  t+2\right)  }{z^{t}}-2%
%TCIMACRO{\dsum \limits_{t=0}^{\infty}}%
%BeginExpansion
{\displaystyle\sum\limits_{t=0}^{\infty}}
%EndExpansion
\frac{y\left(  t+1\right)  }{z^{t}}+%
%TCIMACRO{\dsum \limits_{t=0}^{\infty}}%
%BeginExpansion
{\displaystyle\sum\limits_{t=0}^{\infty}}
%EndExpansion
\frac{y\left(  t\right)  }{z^{t}}+A%
%TCIMACRO{\dsum \limits_{t=0}^{\infty}}%
%BeginExpansion
{\displaystyle\sum\limits_{t=0}^{\infty}}
%EndExpansion
\frac{y\left(  t\right)  }{z^{t}}+B%
%TCIMACRO{\dsum \limits_{t=0}^{\infty}}%
%BeginExpansion
{\displaystyle\sum\limits_{t=0}^{\infty}}
%EndExpansion
\frac{y\left(  t-m\right)  }{z^{t}}=%
%TCIMACRO{\dsum \limits_{t=0}^{\infty}}%
%BeginExpansion
{\displaystyle\sum\limits_{t=0}^{\infty}}
%EndExpansion
\frac{f\left(  t\right)  }{z^{t}},\\
z^{2}\left(
%TCIMACRO{\dsum \limits_{t=0}^{\infty}}%
%BeginExpansion
{\displaystyle\sum\limits_{t=0}^{\infty}}
%EndExpansion
\frac{y\left(  t\right)  }{z^{t}}-\varphi\left(  0\right)  -\frac{1}{z}%
\varphi\left(  1\right)  \right)  -2z\left(
%TCIMACRO{\dsum \limits_{t=0}^{\infty}}%
%BeginExpansion
{\displaystyle\sum\limits_{t=0}^{\infty}}
%EndExpansion
\frac{y\left(  t\right)  }{z^{t}}-\varphi\left(  0\right)  \right) \\
+%
%TCIMACRO{\dsum \limits_{t=0}^{\infty}}%
%BeginExpansion
{\displaystyle\sum\limits_{t=0}^{\infty}}
%EndExpansion
\frac{y\left(  t\right)  }{z^{t}}+A%
%TCIMACRO{\dsum \limits_{t=0}^{\infty}}%
%BeginExpansion
{\displaystyle\sum\limits_{t=0}^{\infty}}
%EndExpansion
\frac{y\left(  t\right)  }{z^{t}}+\frac{B}{z^{m}}\left(  X\left(  z\right)  +%
%TCIMACRO{\dsum \limits_{t=-m}^{-1}}%
%BeginExpansion
{\displaystyle\sum\limits_{t=-m}^{-1}}
%EndExpansion
\frac{\varphi\left(  t\right)  }{z^{t}}\right)  =%
%TCIMACRO{\dsum \limits_{t=0}^{\infty}}%
%BeginExpansion
{\displaystyle\sum\limits_{t=0}^{\infty}}
%EndExpansion
\frac{f\left(  t\right)  }{z^{t}},\\
\left(  \left(  z-1\right)  ^{2}+A+\frac{B}{z^{m}}\right)  X\left(  z\right)
=z\left(  z-1\right)  \varphi\left(  0\right)  +z\Delta\varphi\left(
0\right)  -\frac{B}{z^{m}}%
%TCIMACRO{\dsum \limits_{t=-m}^{-1}}%
%BeginExpansion
{\displaystyle\sum\limits_{t=-m}^{-1}}
%EndExpansion
\frac{\varphi\left(  t\right)  }{z^{t}}+F\left(  z\right)  .
\end{gather*}
This implies%
\begin{gather*}
X\left(  z\right)  =z\left(  z-1\right)  \left(  \left(  z-1\right)
^{2}+A+\frac{B}{z^{m}}\right)  ^{-1}\varphi\left(  0\right)  +z\left(  \left(
z-1\right)  ^{2}+A+\frac{B}{z^{m}}\right)  ^{-1}\Delta\varphi\left(  0\right)
\\
-\left(  \left(  z-1\right)  ^{2}+A+\frac{B}{z^{m}}\right)  ^{-1}%
%TCIMACRO{\dsum \limits_{t=-m}^{-1}}%
%BeginExpansion
{\displaystyle\sum\limits_{t=-m}^{-1}}
%EndExpansion
B\frac{\varphi\left(  t\right)  }{z^{t+m}}+\left(  \left(  z-1\right)
^{2}+A+\frac{B}{z^{m}}\right)  ^{-1}F\left(  z\right)  .
\end{gather*}
In order to get and explicit form of $y\left(  t\right)  ,$we take the inverse
$\mathcal{Z}$-transform, to have%
\begin{align*}
y\left(  t\right)   &  =A_{0}\left(  t\right)  +A_{1}\left(  t\right)  -%
%TCIMACRO{\dsum \limits_{j=-m}^{-1}}%
%BeginExpansion
{\displaystyle\sum\limits_{j=-m}^{-1}}
%EndExpansion
A_{j}\left(  t\right)  +A_{f}\left(  t\right)  ,\\
A_{j}\left(  t\right)   &  =\mathcal{Z}^{-1}\left\{  \frac{1}{z^{j+m}}\left(
\left(  z-1\right)  ^{2}+A+\frac{B}{z^{m}}\right)  ^{-1}B\varphi\left(
j\right)  \right\}  \left(  t\right) \\
&  =\sin(t-j-m-1)B\varphi\left(  j\right)
\end{align*}
where
\begin{align*}
A_{0}\left(  t\right)   &  =\mathcal{Z}^{-1}\left\{  z\left(  z-1\right)
\left(  \left(  z-1\right)  ^{2}+A+\frac{B}{z^{m}}\right)  ^{-1}\varphi\left(
0\right)  \right\}  \left(  t\right)  ,\\
A_{1}\left(  t\right)   &  =\mathcal{Z}^{-1}\left\{  z\left(  \left(
z-1\right)  ^{2}+A+\frac{B}{z^{m}}\right)  ^{-1}\Delta\varphi\left(  0\right)
\right\}  \left(  t\right)  ,\\
A_{j}\left(  t\right)   &  =\mathcal{Z}^{-1}\left\{  \frac{1}{z^{j+m}}\left(
\left(  z-1\right)  ^{2}+A+\frac{B}{z^{m}}\right)  ^{-1}B\varphi\left(
j\right)  \right\}  \left(  t\right)  ,\ \ j\in\mathcal{Z}_{-m}^{-1},\\
A_{f}\left(  t\right)   &  =\mathcal{Z}^{-1}\left\{  \left(  \left(
z-1\right)  ^{2}+A+\frac{B}{z^{m}}\right)  ^{-1}F\left(  z\right)  \right\}
\left(  t\right)  .
\end{align*}
Using the Lemma \ref{lem:2} we get the desired representation (\ref{rep1}).
\end{proof}

\begin{lemma}
Cos$^{A,B}\left(  t\right)  $ and Sin$^{A,B}\left(  t\right)  $ satisfy the
following equations:%
\begin{align}
\Delta\ \text{Cos}^{A,B}\left(  t\right)   &  =-A\ \text{Sin}^{A,B}\left(
t\right)  -B\ \text{Sin}^{A,B}\left(  t-m\right)  ,\ \label{c1}\\
\Delta\ \text{Sin}^{A,B}\left(  t\right)   &  =\text{Cos}^{A,B}\left(
t\right)  ,\label{s1}\\
\Delta^{2}\ \text{Cos}^{A,B}\left(  t\right)   &  =-A\ \text{Cos}^{A,B}\left(
t-1\right)  -B\ \text{Cos}^{A,B}\left(  t-1-m\right)  ,\label{c2}\\
\Delta^{2}\ \text{Sin}^{A,B}\left(  t\right)   &  =-A\ \text{Sin}^{A,B}\left(
t\right)  -B\ \text{Sin}^{A,B}\left(  t-m\right)  . \label{s2}%
\end{align}

\end{lemma}

\begin{proof}
First, we prove the identity (\ref{c1}). It is a consequence of definition of
the determining function $Q\left(  l+1;i\right)  :$
\begin{align*}
\Delta\ \text{Cos}^{A,B}\left(  t\right)   &  =\text{Cos}^{A,B}\left(
t+1\right)  -\text{Cos}^{A,B}\left(  t\right) \\
&  =\sum_{l=1}^{\infty}%
%TCIMACRO{\dsum \limits_{0\leq i\leq l}}%
%BeginExpansion
{\displaystyle\sum\limits_{0\leq i\leq l}}
%EndExpansion
\left(  -1\right)  ^{l}\left[  \left(
\begin{array}
[c]{c}%
t+1-im\\
2l
\end{array}
\right)  -\left(
\begin{array}
[c]{c}%
t-im\\
2l
\end{array}
\right)  \right]  Q\left(  l+1;i\right) \\
&  =\sum_{l=1}^{\infty}%
%TCIMACRO{\dsum \limits_{0\leq i\leq l}}%
%BeginExpansion
{\displaystyle\sum\limits_{0\leq i\leq l}}
%EndExpansion
\left(  -1\right)  ^{l}\left[  \left(
\begin{array}
[c]{c}%
t-im\\
2l-1
\end{array}
\right)  \right]  Q\left(  l+1;i\right) \\
&  =A\sum_{l=1}^{\infty}%
%TCIMACRO{\dsum \limits_{0\leq i\leq l}}%
%BeginExpansion
{\displaystyle\sum\limits_{0\leq i\leq l}}
%EndExpansion
\left(  -1\right)  ^{l}\left[  \left(
\begin{array}
[c]{c}%
t-im\\
2l-1
\end{array}
\right)  \right]  Q\left(  l;i\right) \\
&  +B\sum_{l=1}^{\infty}%
%TCIMACRO{\dsum \limits_{0\leq i\leq l}}%
%BeginExpansion
{\displaystyle\sum\limits_{0\leq i\leq l}}
%EndExpansion
\left(  -1\right)  ^{l}\left[  \left(
\begin{array}
[c]{c}%
t-im\\
2l-1
\end{array}
\right)  \right]  Q\left(  l;i-1\right) \\
&  =-A\sum_{l=0}^{\infty}%
%TCIMACRO{\dsum \limits_{0\leq i\leq l}}%
%BeginExpansion
{\displaystyle\sum\limits_{0\leq i\leq l}}
%EndExpansion
\left(  -1\right)  ^{l}\left[  \left(
\begin{array}
[c]{c}%
t-im\\
2l+1
\end{array}
\right)  \right]  Q\left(  l+1;i\right) \\
&  -B\sum_{l=0}^{\infty}%
%TCIMACRO{\dsum \limits_{0\leq i\leq l}}%
%BeginExpansion
{\displaystyle\sum\limits_{0\leq i\leq l}}
%EndExpansion
\left(  -1\right)  ^{l}\left[  \left(
\begin{array}
[c]{c}%
t-m-im\\
2l+1
\end{array}
\right)  \right]  Q\left(  l+1;i\right) \\
&  =-A\ \text{Sin}^{A,B}\left(  t\right)  -B\ \text{Sin}^{A,B}\left(
t-m\right)  .
\end{align*}
The prove of the identity (\ref{s1}) much more simple:%
\begin{align*}
\Delta\ \text{Sin}^{A,B}\left(  t\right)   &  =\text{Sin}^{A,B}\left(
t+1\right)  -\text{Sin}^{A,B}\left(  t\right) \\
&  =\sum_{l=0}^{\infty}%
%TCIMACRO{\dsum \limits_{0\leq i\leq l}}%
%BeginExpansion
{\displaystyle\sum\limits_{0\leq i\leq l}}
%EndExpansion
\left(  -1\right)  ^{l}\left[  \left(
\begin{array}
[c]{c}%
t+1-im\\
2l+1
\end{array}
\right)  -\left(
\begin{array}
[c]{c}%
t-im\\
2l+1
\end{array}
\right)  \right]  Q\left(  l+1;i\right) \\
&  =\sum_{l=0}^{\infty}%
%TCIMACRO{\dsum \limits_{0\leq i\leq l}}%
%BeginExpansion
{\displaystyle\sum\limits_{0\leq i\leq l}}
%EndExpansion
\left(  -1\right)  ^{l}\left(
\begin{array}
[c]{c}%
t-im\\
2l
\end{array}
\right)  Q\left(  l+1;i\right) \\
&  =\text{Cos}^{A,B}\left(  t\right)  .
\end{align*}
The equations (\ref{c2}) and (\ref{s2}) can be proved by applying (\ref{s1})
and (\ref{c1}).
\end{proof}

The condition $f:\mathbb{Z}_{0}^{\infty}\rightarrow\mathbb{R}^{d}$ is an
exponentially bounded can be eliminated through direct verification, that is
why the proof of the following theorem is not included.

\begin{theorem}
The solution of IVP (\ref{dis1}), (\ref{dis2}) can be rewritten in the f form%
\begin{align}
y\left(  t\right)   &  =\emph{Cos}^{A,B}\left(  t\right)  \varphi\left(
0\right)  +\text{Sin}^{A,B}\left(  t\right)  \Delta\varphi\left(  0\right)  +%
%TCIMACRO{\dsum \limits_{j=-m}^{-1}}%
%BeginExpansion
{\displaystyle\sum\limits_{j=-m}^{-1}}
%EndExpansion
\text{Sin}^{A,B}\left(  t-j-m-1\right)  B\varphi\left(  j\right) \nonumber\\
&  +\sum_{j=0}^{t-2}\text{Sin}^{A,B}\left(  t-j-1\right)  f\left(  j\right)  .
\label{ff1}%
\end{align}

\end{theorem}

\section{Conclusions}

A system of inhomogeneous second-order difference equations with linear parts
given by noncommutative matrix coefficients was considered. The closed-form
solution was derived using newly defined delayed matrix sine/cosine functions
via the $\mathcal{Z}$-transform and determining function. This representation
helped analyze iterative learning control by applying appropriate updating
laws and ensuring sufficient conditions for achieving asymptotic convergence
in tracking.

Future work may focus on controllability, stability, existence and uniqueness
problems of multiple delayed discrete semilinear/linear systems.

\subsection*{}

\section*{Declarations}

\textbf{Conflict of interest} We declare that we do not have any commercial or
associative interest that represents a Conflict of interest in connection with
the work submitted.

\end{document}